\documentclass{amsart}
\usepackage{graphicx}
\usepackage{amssymb}
\usepackage{stackengine,scalerel}
\usepackage{amsthm}
\usepackage{array}
\usepackage{tikz-cd}
\usepackage{enumitem}

\theoremstyle{plain}
\newtheorem{theorem}{Theorem}[section]
\newtheorem{lemma}[theorem]{Lemma}
\newtheorem{prop}[theorem]{Proposition}
\newtheorem{propd}[theorem]{Proposition and Definition}
\newtheorem{cor}[theorem]{Corollary}

\theoremstyle{definition}
\newtheorem{defn}[theorem]{Definition}

\newtheorem{rmk}[theorem]{Remark}
\newtheorem{nota}[theorem]{Notation}

\stackMath
\def\trianglelefteqslant{\ThisStyle{\mathrel{%
  \stackinset{r}{.75pt+.15\LMpt}{t}{.1\LMpt}{\rule{.3pt}{1.1\LMex+.2ex}}{\SavedStyle\leqslant}%
}}}

\def\sS{\mathcal{S}}

\def\bZ{\mathbb{Z}}
\def\bF{\mathbb{F}}

\def\Id{\mathrm{Id}}
\def\Inf{\mathrm{Inf}}

\def\Def{\mathrm{Def}}
\def\dBInf{\mathrm{dBInf}}

\def\Iso{\mathrm{Iso}}
\def\dBIso{\mathrm{dBIso}}

\title{Orthogonal Units of the Double Burnside Ring}
\author{Jamison Barsotti}
\address{Department of Mathematics,
The College of William \& Mary}
\email{jbbarsott@wm.edu}
\begin{document}

\begin{abstract}
Given a finite group $G$, its double Burnside ring $B(G,G)$, has a natural duality operation that arises from considering opposite $(G,G)$-bisets. 
In this article, we systematically study the subgroup of units of $B(G,G)$, where elements are inverse to their dual, so called \emph{orthogonal units}. We show the existence of an
inflation map that embeds the group of orthogonal units of $B(G/N,G/N)$ into the group of orthogonal units of $B(G,G)$, when $N$ is a normal subgroup of $G$,
and study some properties and consequences. In particular, we use these maps to determine the orthogonal units of $B(G,G)$, when $G$ is a cyclic $p$-group, and $p$ is an odd prime.
\end{abstract}

\maketitle
\section{introduction}
The double Burnside ring of a finite group $G$, denoted $B(G,G)$, is an important and interesting invariant in 
the representation theory of finite groups. It is a central object in the study of biset functors, which has in turn answered important questions in representation theory of finite groups.
In particular, bisets functors were used in determining the unit group of the standard Burnside ring for $p$-groups (see \cite{BoucUnits}) and determining the Dade group of a a finite group
(see Chapter $12$ of \cite{BoucBook}). It has also been the subject of study in connection to group fusion and algebraic topology (see \cite{ragnarsson2013saturated}).

Unlike the usual Burnside ring, the double Burnside ring is non-commutative and does not seem to have a convenient, so called, \emph{ghost ring} (see Theorem~\ref{ghostRing}) that one can embed it into.
This has been the subject of much research (see \cite{boltje2012ghost}, \cite{boltje2013ghost}, and \cite{masterson2018table}) . The trade-off is that $B(G,G)$ carries much more data about the group $G$ than its standard Burnside ring, though it is more difficult to work with.

The goal of the article at hand is to begin a structural theory of orthogonal units of the double Burnside ring. That is, units $u\in B(G,G)^\times$ such that $uu^\circ=u^\circ u=\Id_G$, where $(-)^\circ$ is
the natural duality operator on $B(G,G)$ (see Proposition~\ref{dual}) and $\Id_G$ is the identity element of $B(G,G)$. 
These units form a group, denoted by $B_\circ(G,G)$. We offer three main results towards this goal:

First, we introduce inflation maps that embed units of double Burnside rings of quotient groups. If $N$ is a normal subgroup of $G$, we denote these group homomorphism by 
$\dBInf^G_{G/N}:B(G/N,G/N)^\times\to B(G,G)^\times$ (Proposition~\ref{dBInf}). We also define isomorphism maps and show these behave well with the inflation maps, in the biset sense. 
These maps restrict to maps between orthogonal unit groups.

Second, our first main theorem, {\bf Theorem~\ref{mainThm1}}, establishes 
the existence of a naturally occurring elementary abelian $2$-subgroup of orthogonal units of $B(G,G)$. These subgroups have a basis (in the sense of $\bF_2$ vector spaces) parametrized
using the normal subgroups of $G$.

Lastly, our second main theorem (listed below), begins the classification of these orthogonal unit groups for cyclic $p$-groups $C_{p^n}$, where $p$ is a prime. 
We complete this classification for odd primes, leaving only the cases where $p=2$ and $n>1$.

\begin{theorem}\label{main2}
Let $G$ be a cyclic $p$-group with $p$ a prime:
\begin{enumerate}[label=\roman*)]
\item If $G$ is trivial, then $B_\circ(G,G)\cong C_2$.
\item If $G=C_2$, then $B_\circ(G,G)\cong C_2\times D_8$.
\item If $p$ is odd and $|G|=p^n$, then
\[B_\circ(G,G)\cong \left\{\begin{array}{lll}
	C_2^{n+2}\times\prod_{i=1}^n\mathrm{Out}(C_{p^i}) &\mathrm{if}& p=3\\\\
	C_2^{n+1}\times\prod_{i=1}^n\mathrm{Out}(C_{p^i}) &\mathrm{if}& p>3
	\end{array}\right.\]
\end{enumerate}
\end{theorem}

\section{preliminaries}
In this section, we pool together definitions, results, and notation common for the subject. We give direct reference where applicable, and brief proofs for a few of the results. However,
we remark that all the results in this sections can be found in \cite{BoucBook}, \cite{boltje2013ghost}, or \cite{boltje2015orthogonal}.
\subsection{The Burnside ring of a finite group}
Given a finite group $G$, the \emph{Burnside ring} of $G$, which we denote by $B(G)$, is defined to be the Grothendieck ring of isomorphism classes of finite (left) $G$-sets, with 
respect to the operations disjoint union and cartesian product. Recall that the Burnside ring is free as a $\bZ$-module with a basis given by the isomorphism classes of transitive 
$G$-sets. Further, the isomorphism classes of transitive $G$-sets can be parameterized by the subgroups of $G$, up to conjugation, by considering the corresponding coset space. 
So if we let $\sS_G$ be a set of representatives of the conjugacy classes of subgroups of $G$, then we get an explicit basis of $B(G)$, by considering 
$\displaystyle{\{[G/H]\}_{H\in \sS_G}}$, where $[G/H]$ indicates the class of the $G$-set $G/H$ in $B(G)$.

\subsection{The ghost ring of $B(G)$}

If $S$ is a subgroup of $G$ and $X$ is a $G$-set, we can consider the set of elements in $X$, that are fixed by $S$, which we denote by $X^S$. We will the notation $|X^S|$,
to denote the cardinality of this set. For any $g\in G$, we have $|X^{\,^gS}|=|X^S|$. It follows that this induces a ring homomorphism from $B(G)$ into $\bZ$. We
abusively denote the image of this map by $|a^S|$ for $a\in B(G)$. One of the most fundamental results about $B(G)$, due to Burnside, is that these
we can use these maps to embed $B(G)$ into a direct product of $\bZ$. This is called the \emph{ghost ring} of $B(G)$.

\begin{theorem}[Burnside]\label{burnsideTheorem}\label{ghostRing}
Let $G$ be a finite group and $\sS_G$ a set of representatives of conjugacy classes of subgroups of $G$. Then the map
\[B(G)\to \prod_{S\in \sS_G}\bZ\]
\[a\mapsto (|a^S|)_{s\in \bZ}\]
is an injective ring homomorphism with finite cokernel.

\end{theorem}

An immediate corollary to the above fact is that the unit group of $B(G)$ is an elementary abelian $2$-group. However, determining $B(G)^\times$ in general is still very 
open, even for the case of solvable groups. Results and progress on this problem can be found in \cite{BoucUnits}, \cite{YoshidaUnits}, and \cite{Barsotti}.

\subsection{Bisets}

Given finite groups $G$ and $H$, a set $X$ equipped with a left $G$-action and a right $H$-action, such that $(g\cdot x)\cdot h=g\cdot (x\cdot h)$ is called a 
$(G,H)$-\emph{biset}. If we consider the Grothendieck group of finite $(G,H)$-bisets, with respect to disjoint union, these form a $\bZ$ module we denote by $B(G,H)$.
Note that $B(G,H)$ is canonically isomorphic, as a $\bZ$-module, to $B(G\times H)$, since there is a one-to-one correspondence between $(G,H)$-bisets and left $G\times H$-sets
Given by identifying the $(G,H)$ biset $X$ with the left $G\times H$-set $X$, together with the action $(g,h)\cdot x=gxh^{-1}$, for all $(g,h)\in G\times H$ and $x\in X$. 
Thus $B(G,H)$ has a basis given by $\displaystyle{\{[G\times H/L]\}_{L\in \sS_{G\times H}}}$. We recount some fundamental information about bisets.

Recall that a \emph{section} of $G$ is a pair of subgroups $(A,B)$ of $G$, with $B\trianglelefteqslant A$.
Goursat's Lemma gives us an important way to enumerate the standard basis of $B(G,H)$ in terms of sections of $G$ and sections of $H$. 
Given a subgroup $L\leqslant G\times H$, we define the \emph{first and second
projections} of $L$ by $P_{1}(L):=\{g\in G\,|\,(g,h)\in L\}$ and $P_{2(L)}:=\{h\in H\,|\,(g,h)\in L\}$. We also define the \emph{first and second kernels} of $L$ by
$K_{1}(L):=\{g\in G\,|\,(g,1)\in L\}$ and $K_{2}(L):=\{h\in H\,|\,(1,h)\in L\}$. We then have that $K_1(L)\trianglelefteqslant P_1(L)\leqslant G$ and $K_2(L)\trianglelefteqslant P_2(L)\leqslant H$.
In other words, the pairs $(P_1(L),K_1(L))$ and $(P_2(L),K_2(L))$ are sections of $G$ and $H$ respectively. Moreover, there is a canonical isomorphism
$\varphi:P_2(L)/K_2(L)\to P_1(L)/K_1(L)$ such that if $(g,h)\in L$, then $\varphi(gK_2(L))=hK_1(L)$.

\begin{lemma}[Goursat's Lemma, \cite{BoucBook}, Lemma $2.3.25$]\label{GoursatLemma}
If $G$ and $H$ are groups and $L$ is a subgroup of $G\times H$, then there is a unique isomorphism 
$\varphi:P_2(L)/K_2(L)\to P_1(L)/K_1(M)$ such that if $(g,h)\in L$, then $\varphi:(gK_2(L))=hK_1(L)$. 

Conversely, for any sections $(A,B)$ and $(C,D)$ of
$G$ and $H$, respectively, such that there is an isomorphism $\varphi:A/B\to C/D$, there is a subgroup 
\[L_{(A,B),\varphi,(C,D)}=L:=\{(g,h)\in G\times H\,|\, \varphi(hD)=gB\},\]
where $P_1(L)=A,K_1(L)=B, P_2(L)=C$, and $K_2(L)=D$.

\end{lemma}

\begin{rmk}\label{encodeRmk}
If $G$ and $H$ are finite groups and $L\leqslant G\times H$, we frequently identify $L$ with the quintuple $(A,B,\varphi, C,D)$, where $P_1(L)=A$, $P_2(L)=C$, $K_1(L)=B$, $K_2(L)=D$,
and $\varphi$ is the isomorphism $P_2(L)/K_2(L)\overset{\sim}{\to} P_1(L)/K_1(M)$ described by Goursat's Lemma. In this case, we say $L$ is \emph{encoded} as the quintuple $(A,B,\varphi, C,D)$.
We abusively write this as $L=(A,B,\varphi, C,D)$.
We also will write $(A,B,\varphi, C,D)$ as $(A,B;C,D)_\varphi$. In the case where $A=B$ and $C=D$, we will just use the quadruple $(A,B;C,D)$, since there is no choice of $\varphi$.

We also remark that in Propositions~\ref{baseCase} and ~\ref{inductiveCase}, we consider $G$ being a cyclic group of order $p^n$, where $p$ is a prime number. Since the subgroups of $G$
are in one-to-one with the nonnegative integers $0,\cdots, n$, we use the notation $(i,j;k,l)_\varphi$, where $0\leq i,j,k,l\leq n$ and $i-j=k-l\geq0$ to encode subgroups of $G\times G$..

\end{rmk}

If we suppose finite groups $G$, $H$, and $K$, a $(G,H)$-biset $X$, and an $(H,K)$-biset $Y$, we can set 
$X\times_HY$ to be the $H$-obits of the $H$-set $X\times Y$ with the action
$h\cdot(x,y)=(xh^{-1},hy)$, for any $h\in H$ and $(x,y)\in X\times Y$. We then let $X\times_HY$ take on the natural action from $G\times K$. 
Elements of $X\times_HY$ are denoted by $(x,_Hy)$. The operation $\times_H$, which we call the \emph{tensor product} of bisets, induces a bilinear map
from $B(G,H)\times B(H,K)\to B(G,K)$, which we denote with $\circ_H$, such that $[X]\circ_H[Y]=[X\times_HY]$.

\begin{prop}[\cite{BoucBook}, 2.3.14(1)]
Let $G, H, K,$ and $L$ be groups. If $U$ is a $(G,H)$-biset, $V$ is an $(H,K)$-biset, and $W$ is a $(K, L)$-biset. then there is a canonical isomorphism of $(G,L)$-bisets
\[(U\times_HV)\times_KW\cong U\times_H(V\times_KW)\]
given by $((u,_Hv),_Kw)\mapsto (u,_H(v,_Kw))$, for all $(u,v,w)\in U\times V\times W$.
\end{prop}

\begin{rmk}
The above proposition implies that if the bilinear maps $\circ_H:B(G,H)\times B(H,K)\to B(G,K)$ and $\circ_K:B(H,K)\times B(K,L)\to B(H,L)$ interact associatively. That is
\[((a\circ_Hb)\circ_Kc)=(a\circ_H(b\circ_Kc))\in B(G,L),\]
for all $a\in B(G,H)$, $b\in B(H,K)$, and $c\in B(K,L)$.
Because of this, we will frequently write it without the subscript if the context is clear.
\end{rmk}

There is another related operation we need to consider, this time between subgroups of $G\times H$ and $H\times K$. 
Given a subgroup $L \leqslant G\times H$ and a subgroup $M\leqslant H\times K$, we can define a subgroup of $G\times K$ by
\[L*M:=\{(g,k)\in G\times K\,|\, \exists\, h\in H, \mathrm{ \,such\,\, that\, } (g,h)\in L \mathrm{\, and\, } (h,k)\in M\}.\]

The following proposition gives us a formula for computing the product $\circ_H$ of basis elements of $B(G,H)$ and $B(H,K)$ in terms of
the basis of $B(G,K)$.

\begin{prop}[\cite{BoucBook}, $2.3.24$]\label{MackeyFormula}
For $L\leqslant G\times H$ and $M\leqslant H\times K$, we have
\[[G\times H/L]\circ_H[H\times K/M]=\sum_{h\in [P_2(L)\backslash H/P_1(M)]}[G\times K/(L*\,^{(h,1)}M)]\in B(G,K)\]
\end{prop}

\subsection{Opposite bisets}
The following definition is central to our topic.

\begin{defn}(\cite{BoucBook}, $2.3.6$)
If $G$ and $H$ are finite groups and $X$ is a $(G,H)$-biset, then there is a unique $(H,G)$-biset called the \emph{opposite biset} of $X$,
which is equal to the set $X$, equipped with the action
\[h\cdot x\cdot g = g^{-1}xh^{-1} \in X\]
for all $h \in H$, $g\in G$, and $x\in X$. We denote this $(H,G)$-biset by $X^\circ$.
\end{defn}

If $L\leqslant G\times H$, we also consider the \emph{opposite subgroup of $L$}, defined by
\[L^\circ:=\{(h,g)\in H\times G\,|\, (g,h)\in L\}\leqslant H\times G.\]

\begin{lemma}\label{oppositeProp}
Let $G, H$, and $K$ be finite groups.
\begin{enumerate}[label=(\roman*)]
\item If $L$ is a subgroup of $G\times H$, then
\[(G\times H/L)^\circ\cong H\times G/L^\circ\]
as $(H,G)$-bisets.
\item If $L\leqslant G\times H$ and $M\leqslant H\times K$ then
\[(L^\circ)^\circ=L\]
and
\[(L*M)^\circ=M^\circ*L^\circ.\]
\item If $X$ is a $(G,H)$-biset and $Y$ is an $(H,K)$-biset, then
\[(X\times_HY)^\circ\cong Y^\circ\times_HX^\circ\]
as $(K,G)$-bisets.
\item There is a group isomorphism $(-)^\circ:B(G,H)\to B(H,G)$ induced by sending
\[[X]\mapsto[X^\circ]\]
for any $(G,H)$-biset $X$.
\end{enumerate}
\end{lemma}
\begin{proof}
Part $(ii)$ follows easily from definitions. Part $(iv)$ follows from part $(i)$. To prove $(i)$, it is sufficient to check that the map $(g,h)L\to (h^{-1},g^{-1})L^\circ$ is an isomorphism of
$(H,G)$-bisets. Similarly, for $(iii)$, it just needs to be verified that $(x,_Hy)\mapsto (y,_H,x)$ is a well-defined isomorphism of $(K,G)$-bisets. The verifications are straightforward.
\end{proof}

\subsection{Double Burnside rings}

If $G$ is a finite group, then $B(G,G)$ has a ring structure given by the multiplication $[X]\circ_G[Y]=[X\times_GY]$ for all $(G,G)$-bisets $X$ and $Y$.
With this multiplication $B(G,G)$ is called the \emph{double Burnside ring} of $G$. The identity element of $B(G,G)$, which we denote by $\Id_G$, 
is equal to the class $[G]$, where $G$ is the set of elements of $G$ with the usual left and right multiplication as its $(G,G)$-action.

It should be recognized that although $B(G,G)$ and $B(G\times G)$ are isomorphic \emph{as groups}, their ring structures are quite different.
For example, Burnside rings are commutative rings, yet $B(G,G)$ is commutative only when $G$ is trivial. Thus, we have
$B(G,G)\cong B(G\times G)$ as rings if and only if $B(G,G)\cong B(G\times G)\cong B(G)\cong \bZ$.

\begin{prop}\label{dual}
Let $G$ be a finite group. Taking opposite bisets induces an anti-involution on $B(G,G)$. In other words, for any $a,b\in B(G,G)$ we have
\[(a^\circ)^\circ=a\]
and
\[(a\circ_Gb)^\circ=b^\circ\circ_Ga^\circ.\]

\end{prop}
\begin{proof}
The first equality follows from Lemma~\ref{oppositeProp}, using parts $(i), (ii),$ and $(iv)$. The second equality follows from Lemma~\ref{oppositeProp}$(iii)$.
\end{proof}

\subsection{Elementary bisets}

In this subsection, we assume $G$ is a finite group and $N$ is a normal subgroup of $G$. For the topic at hand
we only consider three of the five elementary biset types. The interested reader is encouraged to check out Bouc's treatise on the subject in Chapters $1$ and $2$ from \cite{BoucBook}, where
the following definition and propositions come from.

\begin{defn}
The $(G, G/N)$-biset $G/N$ with natural action will be called \emph{inflation} from $G/N$ to $G$ and denoted by $\Inf_{G/N}^G$.
Dually, we define \emph{deflation} from $G$ to $G/N$ to be the set $G/N$ with natural $(G/N,G)$-action and denote this by $\Def_{G/N}^G$.

If $f:G\mapsto H$ is an isomorphism, then the set $H$ with the $(H,G)$-action $h\cdot x \cdot g=hxf(g)$, for all $h,x\in H$ and all $g\in G$, will be denoted $\Iso(f)$. 
\end{defn}

Note that the biset $\Inf_{G/N}^G$ gives rise to a functor from the category of finite $G/N$-sets to the category of $G$-sets. Similarly, $\Def_{G/N}^G$ gives rise to a functor
from the category of $G$-sets to the category of $G/N$-sets. This justifies the slightly awkward ``from/to" language in their definition.
We abusively denote the images of these bisets in $B(G,G/N)$ and $B(G/N,G)$ as $\Inf_{G/N}^G$ and $\Def_{G/N}^G$, respectively.
Similarly, for an isomorphism $f:G\to H$ we do not notationally distinguish the biset $\Iso(f)$ from its image in $B(H,G)$.

\begin{prop}[\cite{BoucBook}, 1.1.3, 2.b.]\label{IsoInfDef}
Let $G$ be a finite group, $N\trianglelefteqslant G$, and $\varphi:G\to H$ is a group isomorphism., then
\[\Iso(\varphi')\circ\Def^G_{G/N}=\Def_{H/\varphi(N)}^H\circ\Iso(\varphi)\]
\[\Iso(\varphi)\circ\Inf_{G/N}^G=\Inf^H_{H/\varphi(N)}\circ\Iso(\varphi),\]
where $\varphi':G/N\to H/\varphi(N)$ is the gorup isomorphism induced by $\varphi$.

\end{prop}

(Again, all the statements in the proposition below can be found in \cite{BoucBook} Chapters $1$ and $2$, however, 
the last statement of part $(iv)$ in the comes from $6.2.3$ in \cite{BoucBook}.)

\begin{prop}\label{elementaryBisetProp}\hfill
\begin{enumerate}[label=(\roman*)]
\item Identifying $G/1$ with $G$, we have $\Inf_{G/1}^G=\Def_{G/1}^G=\Id_G\in B(G,G)$.
\item We have $(\Inf_{G/N}^G)^\circ=\Def_{G/N}^G\in B(G/N,G)$ and $(\Def_{G/N}^G)^\circ=\Inf_{G/N}^G\in B(G,G/N)$.
\item If $M$ is a normal subgroup of $G$ containing $N$, then
\[\Inf_{G/N}^G\circ\Inf_{G/M}^{G/N}=\Inf_{G/M}^G\in B(G,G/M)\]
and
\[\Def_{G/N}^G\circ\Def_{G/M}^{G/N}=\Def_{G/M}^G\in B(G/M,G).\]
(Note we are identifying $G/M$ canonically with the quotient $(G/N)/(M/N)$.)
\item There is an isomorphism of $(G/N,G/N)$-bisets 
\[\Def_{G/N}^G\times_G\Inf_{G/N}^G\cong G/N.\]
Thus $\Def_{G/N}^G\circ \Inf_{G/N}^G=\Id_{G/N}\in B(G/N,G/N)$. Moreover, $\Inf_{G/N}^G\circ\Def_{G/N}^N$ is an idempotent in $B(G,G)$. 
\end{enumerate}
\end{prop}
\begin{proof}
Statements (i) and (ii) are clear. For (iii), notice that $(G/N)\times_{G/N}(G/M)$ and $G/M$ are isomorphic as 
$(G,G/M)$-bisets via the map $(g_1N,_{G/N}g_2M)\mapsto(g_1g_2M)$. Thus $\Inf_{G/N}^G\cdot_{G/N}\Inf_{G/M}^{G/N}=\Inf_{G/M}^G\in B(G,G/M)$.
Taking opposites give us $\Def_{G/N}^G\circ\Def_{G/M}^{G/N}=\Def_{G/M}^G\in B(G/M,G)$.

For (iv), we use the isomorphism $(g_1N,_{G}g_2N)\mapsto g_1g_2N$ of $(G,G)$-bisets $\Def_{G/N}^G\times_G\Inf_{G/N}^G\cong G/N$. The
last statement follows from the calculation
\[\Inf_{G/N}^G\circ\Def_{G/N}^N\circ\Inf_{G/N}^G\circ\Def_{G/N}^N=\Inf_{G/N}^G\circ\Id_{G/N}\circ\Def_{G/N}^N=\Inf_{G/N}^G\circ\Def_{G/N}^N.\]
\end{proof}

\begin{nota}[\cite{BoucBook}, $6.2.3$]
Given a normal subgroup $N$ of $G$, we use the notation
\[j_N^G:=\Inf_{G/N}^G\circ\Def_{G/N}^N\]
to denote the idempotent from Proposition~\ref{elementaryBisetProp}(iv) associated with $N$.
\end{nota}

One thing to not is that $j_N^G=[G\times G/L]$, where $P_i(L)=G$, and $K_i(L)=N$, with trivial homomorphism $P_1(L)/K_1(L)\cong P_2(L)/K_2(L)$. This tells us that
the set $\{j_N^G\}_{N\trianglelefteqslant G}$ is linearly independent in $B(G,G)$.

The idempotents $j_N^G$ can be used to define a set of useful idempotents from $B(G,G)$. In the definition below 
the function $\mu_{\trianglelefteqslant G}$ denotes the M{\"o}bius function of the poset of normal subgroups of $G$.

\begin{defn}[\cite{BoucBook} $6.2.4$]\label{fidempotentDefinition}
Let $G$ be a finite group and $N\trianglelefteqslant G$. Let $f^G_N$ denote the element in $B(G,G)$ defined by
\[f^G_N:=\sum_{N\leqslant M\trianglelefteqslant G}\mu_{\trianglelefteqslant G}(N,M)j^G_N.\]
\end{defn}

\begin{prop}[\cite{BoucBook} $6.2.5$ and $6.2.7$]\label{fidempotents}
Let $G$. The elements $f_M^G\in B(G,G)$ for $M\trianglelefteqslant G$, are orthogonal idempotents, and for any $N\trianglelefteqslant G$, we have
\[j^G_N=\sum_{N\leqslant M\trianglelefteqslant G}f^G_M.\]
In particular, if $N=1$, we have
\[j_{\{1\}}^G=\Id_G=\sum_{M\trianglelefteqslant G}f^G_M.\]

\end{prop}

\subsection{The $*$ multiplication.}

In the final section, we will need a robust way of computing multiplication in the double Burnside ring. It will be worthwhile to digest the $*$ multiplication a bit more.
We consider the general setting where $G, H$, and $K$ are
finite groups. The following is a classic lemma due to Zassenhaus.

\begin{lemma}[Butterfly Lemma]
Let $(A,B)$ and $(C,D)$ be two sections of $G$. Then there exists a canonical isomorphism 
\[\beta(A',B';C',D'):C'/D'\to A'/B'\]
where $B\leqslant B'\trianglelefteqslant A'\leqslant A$ and $D\leqslant D'\trianglelefteqslant C'\leqslant C$ are defined as
\[ A'=(A\cap C)B,\, B'=(A\cap D)B,\, C'=(C\cap A)D,\,\, \mathrm{and}\,\,\, D'=(C\cap B)D.\]
The isomorphism $\beta(A',B';C',D')$ is uniquely determined by the property that it takes $xD'$ to $xB'$ for all $x\in C\cap A$.
\end{lemma}

Recall Goursat's Lemma (\ref{GoursatLemma}) allows us to describe any subgroup of $G\times H$ as a quintuple $(A,B,\varphi,C,D)$ where the pairs
$(A,B)$ and $(C,D)$ are sections of $G$ and $H$, respectively, and $\varphi:C/D\to A/B$ is a uniquely determined isomorphism. For subgroups
$L\leqslant G\times H$ and $M\leqslant H\times G$, the following lemma describes explicitly the product $L*M$ in these terms. Both the lemma
and subsequent diagram that illustrates it can be found in \cite{boltje2013ghost}.

\begin{lemma}[\cite{boltje2013ghost}, 2.7] \label{explicitStarComputation} 
Let $L=(P_1,K_1, \varphi, P_2,K_2)\leqslant G\times H$ and $M=(P_3,K_3,\psi, P_4,K_4)\leqslant H\times K$. Then
\[L*M=(P_1',K_1',\overline{\varphi}\circ\beta(P_2',K_2';P_3',K_3')\circ\overline{\psi},P_4',K_4')\leqslant G\times K\]
where 
\begin{itemize}
\item $K_2\leqslant K_2'\trianglelefteqslant P_2'\leqslant P_2$ and $K_3\leqslant K_3'\trianglelefteqslant P_3'\leqslant P_3$ are determined by the
Butterfly Lemma applied to the sections $(P_2,K_2)$ and $(P_3,K_3)$ of $H$;
\item $K_1\leqslant K_1'\trianglelefteqslant P_1'\leqslant P_1$ and $K_4\leqslant K_4'\trianglelefteqslant P_4'\leqslant P_4$ are determined by
\[P_1'/K_1=\varphi(P_2'/K_2),\quad K_1'/K_1=\varphi(K_2'/K_2)\]
\[P_4'/K_4=\psi^{-1}(P_3'/K_3),\quad K_4'/K_4=\psi^{-1}(K_3'/K_3);\]
\item the isomorphisms $\overline{\varphi}:P_2'/K_2'\to P_1'/K_1'$ and $\overline{\psi}:P_4'/K_4'\to P_3'/K_3'$ are induced by the isomorphism
$\varphi$ and $\psi$, respectively.
\end{itemize}
\end{lemma}

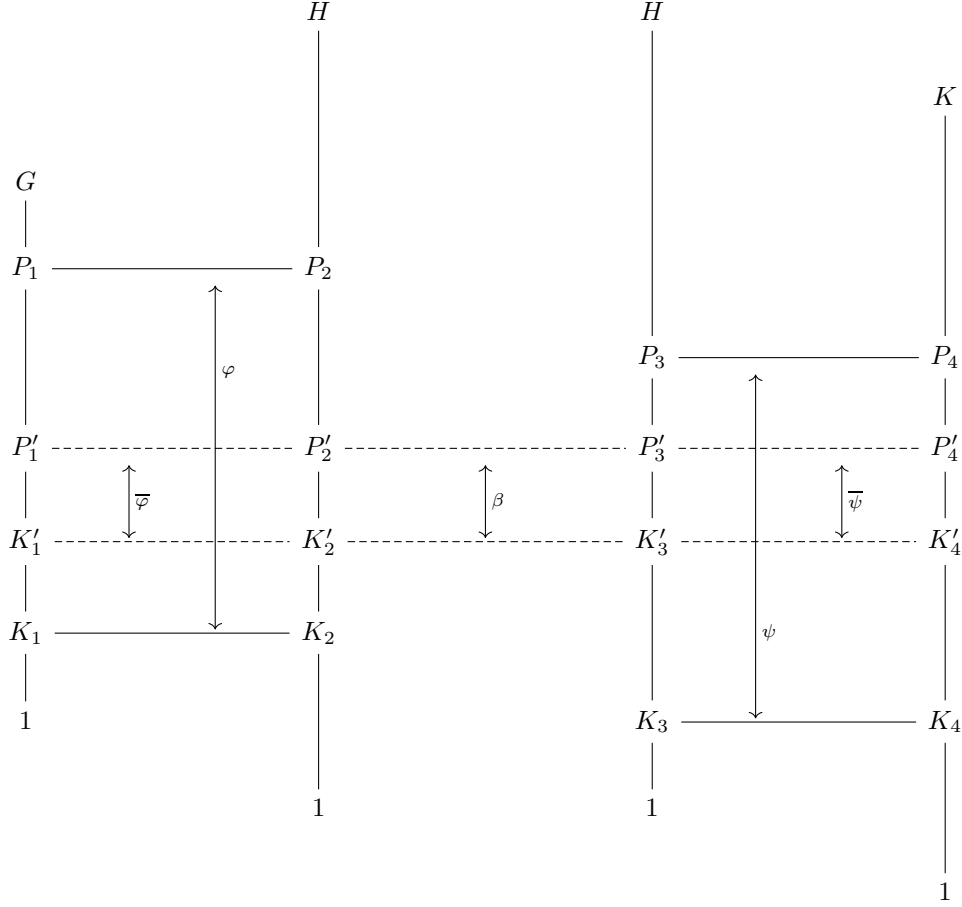
\begin{center}
\begin{tikzcd}
  
  &  &  & H\arrow[ddd, dash]  &  &  &  & H\arrow[dddd, dash] &  &  &  \\

  &  &  &  &  &  &  &  &  &  & K\arrow[ddd, dash] \\

 G\arrow[d, dash] &  &  &  &  &  &  &  &  &  &  \\

 P_1\arrow[rrr, dash]\arrow[dd, dash] &  & |[alias=H_4]| & P_2\arrow[dd, dash] &  &  &  &  &  &  &  \\

  &  &  &  &  &  &  & P_3\arrow[rrr,dash]\arrow[d, dash] & |[alias=C_4]| &  & P_4\arrow[d, dash] \\

 P_1'\arrow[rrr,dash, dashed]\arrow[d, dash] & |[alias=I_4]| &  & P_2'\arrow[rrrr, dash, dashed]\arrow[d, dash] &  & |[alias=E_4]| &  & P_3'\arrow[d, dash]\arrow[rrr,dash, dashed] &  & |[alias=A_4]| & P'_4\arrow[d, dash] \\

 K_1'\arrow[rrr,dash, dashed]\arrow[d, dash] & |[alias=J_4]|\arrow[from=J_4, to=I_4, leftrightarrow, "\overline{\varphi}"']  &  & K_2'\arrow[rrrr, dash, dashed]\arrow[d, dash] &  & |[alias=F_4]|\arrow[from=F_4, to=E_4, leftrightarrow, "\beta"'] &  & K_3'\arrow[dd, dash]\arrow[rrr,dash, dashed] &  & |[alias=B_4]|\arrow[from=B_4, to=A_4, leftrightarrow, "\overline{\psi}"'] & K'_4\arrow[dd, dash] \\

 K_1\arrow[rrr,dash]\arrow[d, dash] &  & |[alias=G_4]|\arrow[from=G_4, to=H_4, leftrightarrow, "\varphi"' near end] & K_2\arrow[dd, dash] &  &  &  &  &  &  &  \\

 1 &  &  &  &  &  &  & K_3\arrow[rrr,dash]\arrow[d, dash] & |[alias=D_4]|\arrow[from=D_4, to=C_4, leftrightarrow, "\psi"' near start] &  & K_4\arrow[dd, dash] \\

  &  &  & 1 &  &  &  & 1 &  &  &  \\

  &  &  &  &  &  &  &  &  &  & 1 
\end{tikzcd}
\end{center}

\subsection{The bifree double Burnside ring}

There are a few notable consequences of Proposition~\ref{MackeyFormula} and Lemma~\ref{explicitStarComputation}. 
Given a finite group $G$ and $M\leqslant G\times G$, if $|P_1(M)/K_1(M)|=|P_2(M)/K_2(M)|<|G|$, then Lemma~\ref{explicitStarComputation} implies 
$|P_1(M*L)/K_1(M*L)|=|P_2(M*L)/K_2(M*L)|<|G|$ and $|P_1(L*M)/K_1(L*M)|=|P_2(L*M)/K_2(L*M)|<|G|$ for all subgroups $L\leqslant G\times G$. Thus,
Proposition~\ref{MackeyFormula} implies that the elements of $B(G,G)$ spanned by the basis elements $[G\times G/M]$ where 
$|P_1(M)/K_1(M)|=|P_2(M)/K_2(M)|<|G|$ form and ideal of $B(G,G)$. 

\begin{prop}[\cite{BoucBook}, $4.3.2$]\label{outGIdeal}
Let $G$ be a finite group and let $I_G$ denote the subgroup of $B(G,G)$ spanned by elements $[G\times G/M]$, where $M\leqslant G\times G$ and
$|P_1(M)/K_1(M)|=|P_2(M)/K_2(M)|<|G|$. Then $I_G$ is an ideal of $B(G,G)$ and there is a surjective ring homomorphism
\[\rho :B(G,G)\to \bZ\mathrm{Out}(G),\]
with $I_G=\ker(\rho)$, that sends $[G\times G/M]\mapsto 0$ if $|P_1(M)/K_1(M)|=|P_2(M)/K_2(M)|<|G|$ and $[G\times G/M]\mapsto \overline{\varphi}$ 
if $P_1(M)=P_2(M)=G$ and 
$K_1(M)=K_2(M)=1$, where $\varphi$ is the uniquely determined automorphism of $G$ indicated by the Goursat Lemma and $\overline{\varphi}$ is its image in 
$\mathrm{Out}(G).$
\end{prop}

\begin{rmk}\label{bifreeBGG}
The map in Proposition~\ref{outGIdeal} is a retraction to the ring homomorphism
\[\eta:\bZ\mathrm{Out}(G)\to B(G,G), \]
\[\overline{\varphi}\mapsto [G\times G/M]\]
where $M\leqslant G\times G$ is defined by the quintuple $(G,1,\varphi,G,1)$. Indeed, the map is well-defined, since the basis
elements of $B(G,G)$ are conjugation invariant. Moreover, if $L=(G,1,\psi,G,1)$, then Proposition~\ref{MackeyFormula} implies
\[[G\times G/M]\cdot_G[G\times G/L]=[G\times G/(M*L)]\]
and Lemma~\ref{explicitStarComputation} implies that $M*L=(G,1,\varphi\circ\psi,G,1)$.
\end{rmk}

If $M\leqslant G\times G$ such that $K_1(M)=K_2(M)=1$, then we say $M$ is \emph{bifree}. Through Goursat's Lemma, bifree subgroups can be identified with
notation $\Delta(A,\varphi,B)$ or $\Delta_\varphi(A,B)$, where $A=P_1(M)\leqslant G$, $B=P_2(M)\leqslant G$ and $\varphi:A\to B$ is an isomorphism. If $A=B$, then we write
$\Delta_\varphi(A,B)=\Delta_\varphi(A)$. In the case where $\varphi$ is the identity, we simply write $\Delta(A)$. Note that $\Delta(A,\varphi,B)=(A,\{1\};B,\{1\})_\varphi$ in the notation of 
Remark~\ref{encodeRmk}.

It is straightforward
to check, using Lemma~\ref{explicitStarComputation}, that for $M,L\leqslant G\times G$, if $K_i(M)=K_i(L)=1$ for $i=1,2$, then $K_i(M*L)=1$ for $i=1,2$.
It follows by Proposition~\ref{MackeyFormula} that the subset $B^\Delta(G,G)\subset B(G,G)$ spanned by the elements $[G\times G/M]$, where $M$ is bifree, is a subring.
We call $B^\Delta(G,G)$ \emph{the bifree double Burnside ring.}

It is clear that for any subgroup $M=(A,B,\varphi,C,D)\leqslant G\times G$, we have $M^\circ=(C,D,\varphi^{-1},A,B)$. Thus if $M$ is bifree, so is $M^\circ$ 
and Proposition~\ref{oppositeProp} implies that taking opposite bisets induces a group automorphism on $B^\Delta(G,G)$.

We end this section with a well-known embedding of $B(G)$ into $B^\Delta(G,G)$. The proof can be found in [\cite{BoucBook}, 2.5.5-2.5.8].

\begin{prop}\label{burnsideRingEmbedding}
Let $G$ be a finite group. The map
\[\iota:B(G)\to B^\Delta(G,G)\]
\[[G/L]\mapsto [G\times G/\Delta(L)]\]
is an injective ring homomorphism.
\end{prop}

\section{An inflation map between units}

We begin this section with an observation that we have a natural embedding of double Burnside rings of quotient groups, in the sense that there exists an injective 
\emph{rng} morphism. Recall, a \emph{rng} is a set with the same properties of a ring, without the assumption of an identity. 
If $A$ and $B$ are rngs, then a \emph{rng morphism} is a map that is both additive and multiplicative. We denote the category of rngs by $\mathrm{{\bf Rng}}$.

\begin{lemma}\label{DBRrng}
Let $G$ be a finite group and $N\trianglelefteqslant G$. Then there is an injective rng morphism
\[B(G/N,G/N)\to B(G,G)\]
\[a\mapsto \Inf^G_{G/N} \circ a\circ \Def^{G}_{G/N}.\]
\end{lemma}
\begin{proof}
The additivity follows from the distributivity of the tensor product of bisets. Recall from 
Proposition~\ref{elementaryBisetProp} $\Def^{G}_{G/N}\circ\Inf^{G}_{G/N}=\Id_{G/N}$,
thus 
\[\Inf^G_{G/N}\circ ab\circ\Def^{G}_{G/N}\]
\[=\Inf^G_{G/N}\circ a\circ\Def^{G}_{G/N}\circ\Inf^G_{G/N}\circ b\circ\Def^{G}_{G/N},\]
 for all $a,b\in B(G/N,G/N)$, so the map is multiplicative. 
 
The injectivity of this map follows from the fact there is a group homomorphism, defined
\[B(G,G)\to B(G/N,G/N)\]
\[x\mapsto \Def^{G}_{G/N}\circ x\circ \Inf^G_{G/N},\]
that is its left inverse. Indeed, we have
\[\Def^{G}_{G/N}\circ\Inf^G_{G/N}\circ a\circ\Def^{G}_{G/N}\circ\Inf^G_{G/N}=a\in B(G/N,G/N).\]
\end{proof}

Let $\mathrm{{\bf Rng}_1}$ by the full subcategory of $\mathrm{{\bf Rng}}$ whose objects are rings (with unity). Below is a generalization
of the familiar functor which restricts rings to their group of units.

\begin{lemma}\label{rngFunctor}
There is a functor 
\[(-)^\times:\mathrm{{\bf Rng}_1}\to \mathrm{{\bf Grp}}\]
defined such that, for any $A\in \mathrm{Ob(\mathrm{{\bf Rng}_1})}$, we have
\[A\mapsto A^\times\]
and for any morphism $f:A\to B$ in $\mathrm{{\bf Rng}_1}$, we have
\[f^\times:A^\times\to B^\times\]
\[u\mapsto1_B+f(u-1_A).\]
Moreover, this functor takes monomorphisms to monomorphisms.
\end{lemma}
\begin{proof}
The last statement is clear. It suffices to check that if $u$ is a unit in $A$, then $1_B+f(u-1_A)$ is a unit in $B$, that $f^\times$ is a group homomorphism, and that
composition is well-defined. All are straightforward but we check composition: If $f:A\to B$ and $g:B\to C$ are morphisms from $\mathrm{{\bf Rng}_1}$ and $u\in A^\times$,
then
\[g^\times\circ f^\times (u)=g^\times (1_B+f(u-1_A))=1_C+g(1_B+f(u-1_A)-1_B)\]
\[=1_C+g(f(u-1_A))=1_C+g\circ f(u-1_A)=(g\circ f)^\times(u).\]
\end{proof}

Using this functor, and Lemma~\ref{DBRrng} we get the following structural maps on unit groups of double Burnside rings.

\begin{propd}\label{dBInf}
Let $G$ be a finite group and $N\trianglelefteqslant G$. Then there is an injective group homomorphism
\[\dBInf_{G/N}^G:B(G/N,G/N)^\times\to B(G,G)^\times\]
defined by
\[u\mapsto 1_G+\Inf_{G/N}^G\circ(u-1_{G/N})\circ\Def_{G/N}^G\]
for all $u\in B(G/N,G/N)^\times$. Moreover, we have
\begin{enumerate}[label=\roman*)]
\item $\dBInf_{G/1}^G$ is the identity map if we identify $G/1$ with $G$, and
\item if $M$ is a normal subgroup of $G$ containing $N$, then
\[\dBInf_{G/N}^G\circ \dBInf_{G/M}^{G/N}=\dBInf_{G/M}^G.\]
Note we are identifying $G/M$ canonically with the quotient $(G/N)/(M/N)$.
\end{enumerate}
\end{propd}
\begin{proof}
The existence of $\dBInf_{G/N}^G$ follows form Lemmas~\ref{DBRrng} and ~\ref{rngFunctor}. The last two properties follow from Proposition~\ref{elementaryBisetProp}.
\end{proof}

The next proposition says that if $N\trianglelefteqslant G$, then the image of the embedding $B(G/N,G/N)\hookrightarrow B(G,G)$ from Lemma~\ref{DBRrng}, can be seen
as the span of basis elements of $B(G,G)$, $[G\times G/L]$ with $L\leqslant G\times G$, which have $N\times N\leqslant L$.

\begin{lemma}\label{kernelArgument}
Let $G$ be a finite group and $N$ a normal subgroup of $G$. Suppose $L\leqslant G\times G$ is the subgroup encoded be Goursat's Lemma as $(P_1,K_1,\varphi, P_2, K_2)$.
If $N\leqslant K_1$ and $N\leqslant K_2$. Define $L'$ to be the subgroup of $G/N\times G/N$ encoded by Goursat's Lemma as
$(P_1',K_1',\overline{\varphi},P_2', K_2')$, where $P_1',P_2', K_1',K_2'$ and $\overline{\varphi}$ are defined respectively by $P_1,P_2, K_1,K_2$ and $\varphi$, through the natural
surjection $G\to G/N$. Then
\[G\times G/L\cong \Inf_{G/N}^G\times_{G/N} (G/N\times G/N)/L'\times_{G/N} \Def_{G/N}^G\]
as $(G,G)$-bisets, via the mapping
\[(g_1,g_2)L\mapsto (N\,,_{G/N}\, (g_1N,g_2N)L'\,,_{G/N}\, N).\]
 
\end{lemma}
\begin{proof}
This amounts to straightforward verification that the explicit map is an isomorphism of bisets.
\end{proof}

We immediately get the following corollary.
\begin{cor}
If $G$ is a finite group and $N$ is a nontrivial normal subgroup of $G$, then 
\[\mathrm{im}(\dBInf_{G/N}^G)\cap B^\Delta(G,G)=\{\Id_G\}.\]
\end{cor}

\begin{nota}
If $f:G\to H$ is an isomorphism of groups, then the map
\[B(G,G)\to B(H,H)\]
\[a\mapsto \Iso(f)\circ a\circ\Iso(f^{-1})\]
is clearly an isomorphism of rings. Moreover, denote the restriction of this map to units by 
\[\dBIso(f):B(G,G)^\times\to B(H,H)^\times.\]
\end{nota}

\begin{prop}
Let $G$ be a finite group and $N$ a normal subgroup of $G$. Suppose $\varphi:G\to H$ is an isomorphism of groups, then
\[\dBIso(\varphi)\circ\dBInf_{G/N}^G=\dBInf_{H/\varphi(N)}^H\circ\dBIso(\varphi'),\]
where $\varphi':G/N\to H/\varphi(N)$ is the isomorphism induced by $\varphi$.
\end{prop}

\begin{proof}
This follows from Proposition~\ref{IsoInfDef}.

\end{proof}

\section{Orthogonal units}

\begin{defn}
Let $G$ be a finite group. A unit $u\in B(G,G)^\times$ is called \emph{orthogonal} if we have
\[uu^\circ=u^\circ u=\Id_G.\]
The set of orthogonal units of $B(G,G)$ is denoted by $B_\circ(G,G)$.
\end{defn}

\begin{rmk}
Given a finite group $G$, the set $B_\circ(G,G)$ of orthogonal units is a subgroup of $B(G,G)^\times$. Indeed, if $u,v\in B_\circ(G,G)$, then by Proposition~\ref{dual}
we have
\[(uv)^\circ=v^\circ u^\circ,\]
thus
\[(uv)(v^\circ u^\circ)=(v^\circ u^\circ)(uv)=\Id_G.\]
So we call $B_\circ(G,G)$ the \emph{group orthogonal units} of $B(G,G)$.
\end{rmk}

\begin{prop}\label{orthogonalRestriction}
Let $G$ be finite group and $N\trianglelefteqslant G$. The map $\dBInf^G_{G/N}$ restricts to a group homomorphism
\[\dBInf^G_{G/N}:B_\circ(G/N,G/N)\to B_\circ(G,G)\]
\end{prop}
\begin{proof} We check that the image of the proposed restriction lands in $B_\circ(G,G)$. 
Let $u\in B_\circ(G/N,G/N)$. Then
\[(\dBInf^G_{G/N}(u))^\circ=[\Id_G+\Inf^G_{G/N}\circ(u-\Id_{G/N})\circ\Def_{G/N}^G]^\circ\]
\[=(\Id_G+\Inf^G_{G/N}\circ u\circ\Def_{G/N}^G-j_N^G)^\circ=\Id_G+\Inf^G_{G/N}\circ u^\circ\circ\Def_{G/N}^G-j_N^G\]
\[=\dBInf^G_{G/N}(u^\circ)=(\dBInf^G_{G/N}(u))^{-1},\]
since $\dBInf^G_{G/N}(u)$ is a group homomorphism.

\end{proof}

From here on, we will assume the function $\dBInf_{G/N}^G$ is the one from Porposition~\ref{orthogonalRestriction}.

Elements from the subset $B^\Delta_{\circ}(G,G):=B_\circ(G,G)\cap B^\Delta(G,G)$ are called \emph{bifree orthogonal units}. Since $M\leqslant G\times G$
is bifree if an only if $M^\circ$ is bifree, it follows that $B^\Delta_{\circ}(G,G)$ is a subgroup of $B_\circ(G,G)$. Boltje and Perepelitsky studied and characterized
these groups for nilpotent $G$.

\begin{theorem}[\cite{boltje2015orthogonal}, 1.1(e)]\label{bifreeNilpotent}
Let $G$ be nilpotent group. Then
\[B^\Delta_{\circ}(G,G)\cong B^\times(G)\rtimes \mathrm{Out}(G)\]
with respect to the natural action of $\mathrm{Out}(G)$ on $B(G)^\times$.
\end{theorem} 

We will need the following result for a detail in Theorem~\ref{mainThm1}. Recall the definition of idempotents $f_N^G\in B(G,G)$, for $N\trianglelefteqslant G$ from Definition~\ref{fidempotentDefinition}.

\begin{lemma}\label{uniqunessLemma}
Let $G$ be a finite group. Let $\mathcal{N}$ and $\mathcal{M}$ be two sets of normal subgroups of $G$. Then
\[\sum_{N\in \mathcal{N}}f_N^G=\sum_{M\in \mathcal{M}}f_M^G.\]
if and only if $\mathcal{N}=\mathcal{M}$.

\end{lemma}
\begin{proof}
The ``if" direction is trivial. The ``only if" direction follows from the fact that the set $\{f_N^G\}_{N\trianglelefteqslant G}$ is linearly independent in $B(G,G)$,
which follows from their definition.
\end{proof}

There is another naturally occurring subgroup of $B_\circ(G,G)$. Trivially, we know that $\Id_G=[G]=[G\times G/\Delta(G)]\in B(G,G)$ is in $B_\circ(G,G)$.
Moreover, we also have that $-\Id_G\in B_\circ(G,G)$. Thus, there is a subgroup obtained by inflating the negative identities, as we run over
all normal subgroups of $G$.

\begin{theorem}\label{mainThm1}
Let $G$ be a finite group. Set $n$ to be the number of normal subgroups $G$ has and
\[H_{dB}:=\langle \dBInf_{G/N}^G(-\Id_{G/N})\,|\, N\trianglelefteqslant G\rangle\]
Then $H_{dB}$ is an elementary abelian $2$-sugroup of $B_{\circ}(G,G)$, with order $2^n$. Moreover, we have
$B^\Delta_{\circ}(G,G)\cap H_{dB}=\{\pm \Id_G\}$.
\end{theorem}

\begin{proof}
We prove this in a slightly indirect fashion. Notice that
\[(\Id_G-2f^G_N)^2=\Id_G-4f^G_N+4f^G_Nf^G_N=\Id_G-4f^G_N+4f^G_N=\Id_G,\]
since $f_N^G$ is an idempotent. Further, we have that
\[(\Id_G-2f^G_N)^\circ=\Id_G^\circ-2(f^G_N)^\circ=\Id_G-2f^G_N,\]
so $\Id_G-2f^G_N\in B_\circ(G,G)$. If we set
\[H'=\langle \Id_G-2f^G_N\,|\, N\trianglelefteqslant G\rangle\]
then we will proceed by proving that $H'$ has all the properties expected of $H_{dB}$ and see that $H_{dB}=H'$.

We first prove that $H'$ is an elementary abelian $2$-group. However, we have seen that every generator of $H'$ has order $2$, so
it suffices to see that it is abelian. If $M$ is a different normal subgroup of $G$, then
\[(\Id_G-2f^G_N)(\Id_G-2f^G_M)=\Id_G-2f_N^G-2f_M^G+4f^G_Nf^G_M\]
\[=\Id_G-2(f_N^G+f_M^G)=\Id_G-2f_N^G-2f_M^G+4f^G_Mf^G_N=(\Id_G-2f^G_M)(\Id_G-2f^G_N),\]
where the second and third equality come from the fact that $f^G_Nf^G_M=f^G_Mf^G_N=0$, which follows from Proposition~\ref{fidempotents} since $N\neq M$.
Moreover, together with Lemma~\ref{uniqunessLemma}, this calculation is easily extended to show every element of $H'$ can be written uniquely as
\[\Id_G-2\left(\sum_{N\in \mathcal{N}}f_N^G\right)\]
where $\mathcal{N}$ is any set of normal subgroups of $G$. Hence $|H'|=2^n$.

By the definition of $\dBInf_{G/N}^G$, we have
\[\dBInf_{G/N}^G(-\Id_{G/N})=\Id_{G}-\Inf_{G/N}^G\circ(\Id_{G/N}+\Id_{G/N})\circ\Def_{G/N}\]
\[=\Id_{G}-2\Inf_{G/N}^G\circ\Def_{G/N}=\Id_{G}-2j^G_N.\]
Thus by Proposition~\ref{fidempotents}, we have
\[\prod_{N\leqslant M\trianglelefteqslant G}(\Id_G-2f_{N}^G)=\Id_G-2\left(\sum_{N\leqslant M\trianglelefteqslant G}f_N^G\right)\]
\[\Id_{G}-2j^G_N=\dBInf_{G/N}^G(-\Id_{G/N}).\]
This proves that $H_{dB}\subseteq H'$. 
Moreover, this calculation shows that working inductively (by descending order, starting with $N=G$), we can replace the generators of $H'$ with the generators of $H_{dB}$,
thus $H_{dB}=H'$.

The last statement comes from noticing that $\displaystyle{\Id_G-2\left(\sum_{N\in \mathcal{N}}f_N^G\right)\in B^\Delta(G,G)}$ if and only if
$\sum_{N\in \mathcal{N}}f_N^G=0$ or $\Id_{G}$, since each $f^G_{N}=\sum_{N\leqslant M\trianglelefteqslant G}\mu_{\trianglelefteqslant G}(N,M)j^G_N$ and
$j_{N}^G\in B^\Delta(G,G)$ if and only if $G=1$. 
Thus $B^\Delta_{\circ}(G,G)\cap H_{dB}=\{\pm \Id_G\}$.
\end{proof}

\begin{cor}\label{proper}
Let $G$ be a finite group. Then $B^\Delta_\circ(G,G)=B_\circ(G,G)$ if and only if $G$ is trivial.
\end{cor}

\begin{proof}
In the case where $G$ is trivial, it is easy to see that $B(G,G)\cong B(G)$ as rings and that $B(G,G)^\times=B_\circ(G,G)=B^\Delta_\circ(G,G)=\{\pm\Id_G\}$.
Otherwise, Theorem~\ref{mainThm1} shows that $H_{dB}\subset B_{\circ}(G,G)$ 
has at least $4$ elements but $B^\Delta_{\circ}(G,G)\cap H_{dB}=\{\pm \Id_G\}$. The result follows.
\end{proof}

\begin{rmk}
The genesis of this paper began when the author's advisor, Robert Boltje, asked the author to investigate orthogonal units of the double Burnside ring that are not bifree. 
There is a connection to modular representation theory in considering what is called the \emph{trivial source ring}. If $F$ is an algebraically closed field of characteristic $p>0$ and
$G$ and $H$ are finite groups with blocks $A$ and $B$ from $FG$ and $FH$, repspectively, we denote by $T(A,B)$ to be the Grothendieck group, with respect to direct sum, of 
$(A,B)$-bimodules that are direct summands of finitely generated permutation $F(G\times H)$ modules. If $G=H$ and $A=B$, this is a ring with respect to the tensor product over $A$ (or over $FG$). 

Taking $F$-duals
gives rise to a bilinear group isomorphism from $T(A,A)$ to itself, $\gamma\mapsto \gamma^\circ$, with the property $(\gamma\beta)^\circ=\beta^\circ\gamma^\circ$.
In \cite{boltje2015orthogonal}, it is posed to consider the group of auto-equivalencies of the subgroup $T^\Delta(A,A)\subset T(A,A)$, with respect to taking duals, that
is elements $\gamma\in T^\Delta(A,A)$ such that $\gamma\gamma^\circ=\gamma^\circ\gamma=\Id$, where $T^\Delta(A,A)$ is the the subgroup of $T(A,A)$ spanned by
those standard basis elements of $T(A,A)$ which have vertex coming from a subgroup of the form $\Delta_\varphi(X,Y)$ of $G\times G$. This group is denoted by $T^\Delta_\circ(A,A)$.
However, it makes sense to also consider the group $T_\circ(A,A)$, i.e. all elements $\gamma\in T(A,A)$, such that $\gamma\gamma^\circ=\gamma^\circ\gamma=\Id$.

In the case that $G$ is a $p$-group, then $A=FG$ and there is a canonical, dual preserving, isomorphism $T(A,A)\cong B(G,G)$, that restricts to an isomorphism
$T^\Delta(A,A)\cong B^\Delta(G,G)$. In particular, Corollary~\ref{proper} can be used to show that in general $T^\Delta_\circ(A,A)$ is a proper subgroup
of $T_\circ(A,A)$.

More information on $T(A,B)$ and $T^\Delta_\circ(A,A)$ can be found in \cite{perepelitsky2014p}.
\end{rmk}

\begin{lemma}[\cite{boltje2015orthogonal}, 3.2(c)]\label{thing}
Let $G$ be a finite group. For each $\gamma\in B_\circ^\Delta(G,G)$, there is a unique $\overline{\varphi}\in \mathrm{Out}(G)$ and a unique $\epsilon\in \{\pm1\}$ such that 
$\rho(\gamma)=\epsilon\overline{\varphi}$. Moreover, the resulting map
\[B_\circ^\Delta(G,G)\to \mathrm{Out}(G)\]
\[\gamma\mapsto \overline{\varphi}\]
is a surjective group homomorphism.
\end{lemma}

\begin{rmk}\label{variation}
We make a slight variation on the above map to fit better with our purposes. If we identify $\mathrm{Out}(G)$ with its image in $(\bZ\mathrm{Out}(G))^\times$ and consider the the subgroup
$\langle-\Id_{\mathrm{Out}(G)},\mathrm{Out}(G)\rangle \leqslant (\bZ\mathrm{Out}(G))^\times$. The above lemma tells us that restricting the the map $\rho$ gives us a surjective
group homomorphism
\[\rho^\times: B_\circ^\Delta(G,G)\to\langle-\Id_{\mathrm{Out}(G)},\mathrm{Out}(G)\rangle.\]
\end{rmk}

The first part of next lemma shows that we can extend the map from Remark~\ref{variation} to all of $B_\circ(G,G)$. All parts are likely known by experts, except for the last part.

\begin{lemma}\label{workhorse}
Let $G$ be a finite group.

\begin{enumerate}[label=(\roman*)]
\item For each $\gamma\in B_\circ(G,G)$, there is a unique $\overline{\varphi}\in\mathrm{Out}(G)$ and a unique $\epsilon\in\{\pm1\}$ such that
$\rho(\gamma)=\epsilon\overline{\varphi}$
In particular, the map $\rho:B(G,G)\to \bZ\mathrm{Out}(G)$ restricts to a surjective group homomorphism
\[\rho^\times:B_\circ(G,G)\to \langle-\Id_{\mathrm{Out}(G)}, \mathrm{Out}(G)\rangle\]
\[u\mapsto \mathrm{sgn}(\epsilon)\overline{\varphi},\]
where we are identifying the group $\mathrm{Out}(G)$ with its image in $(\bZ\mathrm{Out}(G))^\times$.
\item The map $\eta:\bZ\mathrm{Out}(G)\to B(G,G)$ restricts to an injective group homomorphism
\[\eta:\langle -\Id_{\mathrm{Out}(G)},\mathrm{Out}(G)\rangle\to B_\circ(G,G),\]
such that $\rho^\times\circ\eta=\Id$. In particular, $B_\circ(G,G)=\mathrm{im}(\eta)\ltimes\ker(\rho^\times)$. 
\item There is a one-to-one correspondence between $u\in \ker(\rho^\times)$ and elements $a\in I_G$ such that
\[aa^\circ=a^\circ a=a+a^\circ.\]
\item If $N$ is a nontrivial normal subgroup of $G$, then $\mathrm{im}(\dBInf_{G/N}^G)\leqslant \ker(\rho^\times)$.
\end{enumerate}
\end{lemma}
\begin{proof}
To prove $(i)$ let $u\in B_\circ(G,G)$ and write $u=u_\Delta+u_I$ where 
\[u_\Delta=\sum_{\overline{\varphi}\in \mathrm{Out}(G)}c_{\varphi}[G\times G/\Delta_\varphi(G)]\]
 with
$c_{\varphi}\in \bZ$ for all $\overline{\varphi}\in \mathrm{Out}(G)$, and $u_I\in I_G$. Then we have $u^\circ=u_\Delta^\circ+u_I^\circ$, with $u_I^\circ\in I_G$ and 
\[u_\Delta^\circ=\sum_{\overline{\varphi}\in \mathrm{Out}(G)}c_{\varphi}[G\times G/\Delta_{\varphi^{-1}}(G)].\]
We prove that $u_\Delta\in B_\circ^\Delta(G,G)$ and the result will follow
from Lemma~\ref{thing} since $\rho(u_I)=0$. We have that $u_\Delta\in B^\Delta(G,G)$ so it suffices to see that $u_\Delta u_\Delta^\circ=u_\Delta^\circ u_\Delta=\Id_G$. Indeed,
\[\Id_G=uu^\circ=(u_\Delta+u_I)(u_\Delta^\circ+u_I^\circ)=u_\Delta u_\Delta^\circ+u_\Delta u_I^\circ+u_Iu_\Delta ^\circ+u_Iu_I^\circ.\]
Yet, $u_\Delta u_I^\circ+u_Iu_\Delta ^\circ+u_Iu_I^\circ\in I_G$ and if we write $u_\Delta u_\Delta^\circ$ in terms of the standard basis elements of $B(G,G)$, none of the summands will be in $I_G$, hence
$u_\Delta u_I^\circ+u_Iu_\Delta ^\circ+u_Iu_I^\circ=0$ and so $u_\Delta u_\Delta^\circ=\Id_G$. Similarly, we get that $u_\Delta^\circ u_\Delta=\Id_G$.

Part $(ii)$ is clear from the definition of $\eta$.

Part $(iii)$ follows by writing $u=\Id_G-a$ and noticing that $u\in \ker(\rho^\times)$ if and only if $\rho(a)=0$ and 
\[\Id_G=uu^\circ=\Id_G-a-a^\circ+aa^\circ=u^\circ u=-a-a^\circ+a^\circ a,\]
if and only if $a\in I_G$ and 
\[aa^\circ=a^\circ a=a+a^\circ.\]

Part $(iv)$ follows from part $(iii)$ and Lemma~\ref{kernelArgument}.

\end{proof}

\begin{rmk}
There is another way to naturally produce units in $B_\circ(G,G)$, namely via the embedding $\iota:B(G)\to B(G,G)$ (see Proposition~\ref{burnsideRingEmbedding}). 
In fact, if we restrict this map to units we get a map
\[\iota:B(G)^\times\to B^\Delta_\circ(G,G).\]
Moreover, if we look at the subgroup of $B(G)^\times$ consisting of elements $x\in B(G)^\times$ such that $|x^G|=1$ (see Theorem~\ref{burnsideTheorem}),
this can be identified with $\overline{B}(G)^\times:=B(G)^\times/\{\pm 1\}$, and $\iota$ induces an injective group homomorphism
\[\iota':\overline{B}(G)^\times\to B^\Delta_\circ(G,G)\cap \ker(\rho^\times).\]
That this map is surjective for nilpotent groups follows from Lemma~\ref{workhorse} and Theorem~\ref{bifreeNilpotent}. However, it is shown in \cite{boltje2015orthogonal} ($4.1$, $4.3$)
that $\iota'$ is not surjective in general.

Furthermore, if $N$ is a nontrivial normal subgroup of $G$, we also have $\mathrm{im}(\iota')\cap\mathrm{im}(\dBInf_{G/N}^G)=\{\Id_G\}$. This follows
since $\mathrm{im}(\dBInf_{G/N}^G)\cap B^\Delta(G,G)=\Id_G$, which is a consequence of Lemma~\ref{kernelArgument}.

\end{rmk}

\section{Cyclic $p$-groups}

In this final section, we use the inflation maps between units of double Burnside rings to prove Theorem~\ref{main2}. If $G$ is a finite group and
$N$ is a normal subgroup of $G$, we will assume $\dBInf_{G/N}^G$ is the map from $B_\circ(G/N,G/N)$ to $B_\circ(G,G)$ established in Proposition~\ref{orthogonalRestriction}.

Since we are working with double Burnside rings of cyclic groups it is useful to consider double Burnside rings for general abelian groups. In particular, 
we study a useful isomorphism for calculation in the double Burnside ring in this case. Before we do so, suppose
$G$ is an abelian group and let $\sS_{G\times G}$ denote the set of subgroups of $G\times G$. Define the map
\[\gamma:\sS_{G\times G}\times \sS_{G\times G}\to \bZ\]
\[\gamma(L,M)=\frac{|G|}{|P_2(L)P_1(M)|}.\]
Notice that since $G$ is abelian, Proposition~\ref{MackeyFormula} tells us the product 
\[[G\times G/L]\circ_G[G\times G/M]=\sum_{h\in P_1(L)\backslash G/P_2(M)}[G\times G/(L*\,^{(h,1)}M)]\]
\[=\sum_{h\in G/(P_1(L)P_2(M))}[G\times G/(L*M)]=\gamma(L,M)[G\times G/(L*M)].\]
It follows from the associativity of $\circ_G$ that $\gamma$ satisfies the $2$-cocycle relation.

\begin{defn}
If $G$ is a finite abelian group and $\sS_{G\times G}$ is the set of subgroups of $G\times G$. We define $\bZ_\gamma\sS_{G\times G}$ to be the 
$\bZ$-algebra with basis given by the elements of $\sS_{G\times G}$ and multiplication defined by extending
\[L*'M:=\gamma(L,M)L*M,\]
for $L,M\in \sS_{G\times G}$, linearly to all elements of $\bZ_\gamma\sS_{G\times G}$.
\end{defn}

\begin{prop}\label{abelianIsomorphism}
Let $G$ be an abelian group. We have an isomorphism of algebras $B(G,G)\cong\bZ_\gamma\sS_{G\times G}$ given by the map
\[[G\times G/M]\mapsto M.\]
Moreover, the duality operator on $B(G,G)$ corresponds with the duality operator on $\bZ_\gamma\sS_{G\times G}$ induced by taking opposite bisets.
\end{prop}
\begin{proof}
Since $G$ is abelian, this is a one-to-one correspondence between bases. The verification that multiplication is preserved follows from Proposition~\ref{MackeyFormula}.
The last statement follows from Proposition~\ref{oppositeProp}.
\end{proof}

In the following proofs, we abusively identify $B(G,G)$ with $\bZ_\gamma\sS_{G\times G}$ , since $G$ will always be abelian. We also ignore the operator $*'$.

\begin{lemma}\label{trivialAction}
Suppose $G$ is finite a cyclic group. Then
\[B_\circ(G,G)=\mathrm{im}(\eta)\times\ker(\rho^\times)\]
where $\eta$ and $\rho^\times$ are the maps from Lemma~\ref{workhorse}.
\end{lemma}
\begin{proof}
This amounts to showing that conjugation by an element of $\mathrm{im}(\eta)$ is trivial on elements of $\ker(\rho^\times)$. We can actually say more. 
In fact, we show that $\mathrm{im}(\eta)\subset Z(B(G,G))$. Every element of $\mathrm{im}(\eta)$ is of the form $\epsilon\Delta_\varphi(G)$ where $\epsilon\in \{\pm 1\}$
and $\varphi$ is an automorphism of $G$. Write $G=\langle x\rangle$, then $\varphi(x)=x^k$ where $k$ relatively prime with the order of $G$. Suppose
$L=(P_1,K_1,\psi,P_2,K_2)\in \sS_{G\times G}$. By Proposition~\ref{abelianIsomorphism} and Lemma~\ref{explicitStarComputation} we have
\[\Delta_\varphi(G)*'L=(G,1,\varphi,G,1)*'(P_1,K_1,\psi,P_2,K_2)=(P_1,K_1,\overline{\varphi}\circ\psi,P_2,K_2)\]
where $\overline{\varphi}:P_1/K_1\to P_1/K_1$ is the map that takes $xK_1$ to $x^kK_1$. However, $\overline{\varphi}\circ\psi=\psi\circ\overline{\varphi}'$ where
$\overline{\varphi}':P_2/K_2\to P_2/K_2$ that takes $xK_2\mapsto x^kK_2$. Thus
\[\Delta_\varphi(G)*'L=(P_1,K_1,\overline{\varphi}\circ\psi,P_2,K_2)=(P_1,K_2,\psi\circ\overline{\varphi}',P_2,K_2)=L*'\Delta_\varphi(G),\]
where the last equality, again, comes from Lemma~\ref{explicitStarComputation}. The result follows.
\end{proof}

We now specialize to the case where $p$ is a prime and $G$ is a cyclic $p$-group.

Theorem~\ref{main2} will be a consequence of the next two propositions. It is proved by induction. The first proposition encompasses the base case, with the next proposition essentially
being the inductive step when $p$ is odd.  We refer the reader to Remark~\ref{encodeRmk} for a recap on the notation used for the following propositions.

\begin{prop}\label{baseCase}
Let $G=C_p$ where $p$ is a prime.
\begin{enumerate}[label=(\roman*)]
\item If $p=2$, then $B_\circ(G,G)\cong C_2\times D_8$.
\item If $p=3$, then $B_\circ(G,G)\cong C_{p-1}\times C_2^3$
\item If $p>3$, then $B_\circ(G,G)\cong C_{p-1}\times C_2^2$
\end{enumerate}
\end{prop}

\begin{proof}

Referring to Lemma~\ref{workhorse}, we have that $\mathrm{im}(\eta)\cong C_{p-1}\times C_2$ and Lemma~\ref{trivialAction} shows that $\mathrm{im}(\eta)$ is in the center of $B(G,G)$. 
What is left is to determine $\ker(\rho^\times)$.

Suppose $u\in \ker(\rho^\times)$. By Lemma~\ref{workhorse} $(iii)$, we can write $u=\Id_G-\alpha$ where $\alpha\in I_G$ and $\alpha\alpha^\circ=\alpha^\circ \alpha=\alpha+\alpha^\circ$. 
Since $G$ has exactly two
subgroups, it follows by Goursat's Lemma that $I_G$ is spanned by exactly four elements, namely $w=(1,1,\Id,1,1), x=(1,1,\Id,G,G), y=(G,G,\Id,1,1),$ and $z=(G,G,\Id,G,G)$.

Notice that $w^\circ=w$ and $z^\circ=z$ and $x^\circ=y$ and $y^\circ=x$. Given integers $a_1,a_2,a_3,a_4\in \bZ$, we can write
\[\alpha=a_1w+a_2x+a_3y+a_4z\]
and
\[\alpha^\circ=a_1w+a_3x+a_2y+a_4z.\]
It is straightforward to verify, using Proposition~\ref{abelianIsomorphism} and Lemma~\ref{explicitStarComputation}, that 
\[ww=pw, wx=px, wy= w, wz=x,\]
\[xw=w, xx=x, xy= w, xz=x,\]
\[yw=py, yx=pz, yy= y, yz=z,\]
\[zw=y, zx=z, zy= y, zz=z.\]
If we write
\[\alpha\alpha^\circ=c_1w+c_2x+c_3y+c_4z,\]
with
\[c_1=pa_1^2+2a_1a_2+a_2^2=(p-1)a_1^2+(a_1+a_2)^2\]
\[c_2=pa_1a_3+a_1a_4+a_2a_3+a_2a_4\]
\[c_3=pa_1a_2+a_2a_3+a_1a_4+a_3a_4\]
\[c_4=pa_3^2+2a_3a_4+a_4^2=(p-1)a_3^2+(a_3+a_4)^2,\]
and note that $\alpha+\alpha^\circ=2a_1x+(a_2+a_3)y+(a_2+a_3)w+2a_4z$, we have

\[c_1=(p-1)a_1^2+(a_1+a_2)^2=2a_1\]
\[c_2=pa_1a_3+a_1a_4+a_2a_3+a_2a_4=a_2+a_3\]
\[c_3=pa_1a_2+a_2a_3+a_1a_4+a_3a_4=a_2+a_3\]
\[c_4=(p-1)a_3^2+(a_3+a_4)^2=2a_4.\]
Dually, we have
\[\alpha^\circ \alpha=d_1w+d_2x+d_3y+d_4z=\alpha+\alpha^\circ,\]
which implies
\[d_1=(p-1)a_1^2+(a_1+a_3)^2=2a_1\]
\[d_2=pa_1a_2+a_1a_4+a_2a_3+a_3a_4=a_2+a_3\]
\[d_3=pa_1a_3+a_2a_3+a_1a_4+a_2a_4=a_2+a_3\]
\[d_4=(p-1)a_2^2+(a_2+a_4)^2=2a_4.\]
Thus, our search boils down to finding quadruples $(a_1,a_2,a_3,a_4)$ of integers that satisfy the above quadratic equations.

Notice that $c_1, d_1\geq0$, this implies that $a_1\geq 0$. Moreover, from $c_1$ and $d_1$ we also have that 
\[(a_1+a_2)^2=(a_1+a_3)^2=(2-(p-1)a_1)a_1\geq 0.\]
If $a_1=0$, then $a_2=a_3=0$. Looking at $c_4$ and $d_4$, this leave $a_4=0$ or $a_4=2$. Note that both $(0,0,0,0)$ and $(0,0,0,2)$ satisfy our system of equations.

If $a_1\neq0$, then this implies that $2-(p-1)a_1\geq0\implies \frac{2}{p-1}\geq a_1$. Since $p$ is a prime, this forces $p=2$ or $p=3$. We note that this implies that if
that in the case $p>3$, $a_1=0$ and $\ker(\rho^\times)$ has
exactly $2$ elements, thus it is isomorphic to $C_2$ and this proves $(iii)$. We split the rest of the proof up into the two obvious cases.\\

{\bf Case $p=3$:}

We continue with the assumption that $a_1>0$. The inequality $\frac{2}{p-1}\geq a_1$ forces $a_1=1$. The coefficients $c_1$ and $d_1$ then imply that
\[(1+a_2)^2=(1+a_3)^2=0\implies a_2=a_3=-1\]
Looking at the coefficients $c_4$ and $d_4$, we can conclude that $a_4$ must satisfy the quadratic equation
\[2+(a_4-1)^2=2a_4.\]
Thus $a_4=1$ or $a_4=3$. Checking that the quadruples $(1,-1,-1,1)$ and $(1,-1,-1,-3)$ both satisfy the equations given by the coefficients $c_2$ and $c_3$ (note that $d_2$ and $d_3$ are the same).
We see that $|\ker(\rho^\times)|=4$. That it is isomorphic to $C_2^2$ comes form the fact that every element is self dual, thus has order $2$. This proves $(ii)$.\\

{\bf Case $p=2$:}
We again assume that $a_1>0$ and call upon the inequality $\frac{2}{p-1}\geq a_1$. There are two cases, either $a_1=1$ or $a_1=2$. If $a_1=2$,
then $c_1$ and $d_1$ imply that $(2+a_2)^2=(2+a_3)^2=0$, which forces $a_2=a_3=-2$. Using the coefficients $c_4$ and $d_4$, this means that $a_4$ must satisfy the quadratic
$4+(a_4-2)^2=2a_4$ and this implies $a_4=2$ or $a_4=4$. Note that both the quadruples $(2,-2,-2,2)$ and $(2,-2,-2,4)$ satisfy the equations given by the coefficients $c_2$ and $c_3$.

If $a_1=1$, then $(1+a_2)^2=(1+a_3)^2=1\implies a_2\in\{0,-2\}$ and $a_3\in\{0,-2\}$. If $a_2=a_3=0$ then any of the equations provided by $c_2,c_3,d_2,d_3$ imply that $a_4=0$. Clearly $(1,0,0,0)$
satisfies $c_4$ and $d_4$. 

If $a_2=-2$ (respectively $a_3=-2$), then $d_4$ (respectively $c_4$) implies $a_4=2$ or $a_4=4$. Which narrows the other quadruples down to $(1,-2,0,2), (1,-2,0,4),$ $(1,0,-2,2),$ 
$(1,0,-2,4), (1,-2,-2,2),$ and $(1,-2,-2,4)$. Notice that the only quadruples that satisfy the coefficients $c_2$ and $c_2$, are $(1,-2,0,2)$, $(1,0,-2,2)$ and $(1,-2,-2,4)$.
Thus, $\ker(\rho^\times)$ is a group of order $8$, parametrized by the quadruples
\[(0,0,0,0),(0,0,0,2),(2,-2,-2,2), (2,-2,-2,4),\]
\[(1,0,0,0),(1,-2,0,2),(1,0,-2,2),(1,-2,-2,4).\]
Notice that exactly $6$ elements are self dual. This implies that $\ker(\rho\times)$ has $5$ elements of order $2$ and $2$ elements of order $4$. 
Thus, $\ker(\rho^\times)\cong D_8$
proving $(i)$.
\end{proof}

We note that the above proposition gives an outlines for how to find orthogonal units of double Burnside rings for a general finite group $G$. 
Boiling the process down to solving a system of (several) quadratic equations.
However, as can already be seen, this process is rather tedious and not particularly insightful. 
For cyclic groups of order $p$, it showcases that $p=2$ and $p=3$ are exceptional cases. Yet, in the inductive step, we will see that $p=3$ falls in line with the rest of the odd cases. 
However, the case where $p=2$ remains exceptional. 
We leave the $p=2$ open to further research for now.\\

\begin{prop}\label{inductiveCase}
Let $p$ be an odd prime. Let $G=C_{p^n}$ with $n>1$. Then
\[B_\circ(G,G)=\mathrm{im}(\eta)\times \mathrm{im}(\dBInf_{C_{p^n}/C_p}^{C_{p^n}}).\]
\end{prop}

\begin{proof}
Our strategy starts off similarly to how we approached Proposition~\ref{baseCase}. Using Lemma~\ref{workhorse} $(ii)$ and Lemma~\ref{trivialAction}, we want to show that 
$\mathrm{im}(\rho^\times)=\mathrm{im}(\dBInf_{C_{p^n}/C_p}^{C_{p^n}})$. To accomplish this let $u=\Id_G-\alpha\in \mathrm{im}(\rho^\times)$
with $\alpha\in I_G$ and $\alpha\alpha^\circ=\alpha^\circ \alpha=\alpha+\alpha^\circ$. Moreover, if $\sS_{G\times G}$ is the set of subgroups of $G\times G$, we can write
\[\alpha=\sum_{X\in \sS_{G\times G}}a_XX\]
with $a_X\in \bZ$. By Goursat's Lemma, each $X$ can be encoded as a quintuple $(C_{p^a},C_{p^b},\varphi, C_{p^c},C_{p^d})$ where $0\leq a,c\leq n$, $a-b=c-d>0$ and $\varphi$
is an isomorphism $C_{p^c}/C_{p^d}\cong C_{p^a}/C_{p^b}$. We abbreviate this by $(a,b;c,d)_\varphi$. Our goal is to show that if $b=0$ or $d=0$, then $a_X=0$ for all $X\in \sS_{G\times G}$.
By Lemma~\ref{kernelArgument}, this will imply that there is some $\alpha'\in B(C_{p^n}/C_p,C_{p^n}/C_p)$ such that 
$\alpha=\Inf_{C_{p^n}/C_p}^{C_{p^n}}\circ \alpha'\circ \Def_{C_{p^n}/C_p}^{C_{p^n}}$, where $\alpha'\alpha'^\circ=\alpha'^\circ \alpha'=\alpha'+\alpha'^\circ$.
In other words, $u\in \mathrm{im}(\dBInf_{C_{p^n}/C_p}^{C_{p^n}})$. The result then follows from Lemma~\ref{workhorse} $(iv)$.\\

We begin by writing
\[\alpha\alpha^\circ=\sum_{X\in \sS_{G\times G}}c_XX\]
We will work inductively, first by considering the coefficient $c_{\Delta(C_{p^{n-1}})}$. 
Recall that $\Delta(C_{p^{i}})=(i,0;i,0)\in \sS_{G\times G}$.
Notice that for any $X,Y\in \sS_{G\times G}$, with $X,Y\in I_G$ and $X*Y=\Delta(C_{p^{n-1}})$, 
Lemma~\ref{explicitStarComputation} implies that $X$ is encoded as $(n-1,0;c,d)_\varphi$ and $Y$ is encoded as $(c,d;,n-1,0)_{\varphi^{-1}}$, where $c-d=n-1$. 
In other words, $X^\circ=Y$. This implies that the coefficient $c_{\Delta(C_{p^{n-1}})}$ of $\alpha\alpha^\circ$, is equal to
\[pa_{\Delta(C_{p^{n-1}})}^2+p\sum_{Y}a_Y^2+\sum_{Z}a_Z^2\]
where $Y$ runs over all the elements of $\sS_G$ encoded as $(n-1,0;n-1,0)_\varphi$ with $\varphi$ nontrivial, and $Z$ runs over the elements of $\sS_G$ encoded as $(n-1,0;n,1)_{\psi}$. Since 
$\alpha \alpha^\circ=\alpha+\alpha^\circ$, we have $c_{\Delta(C_{p^{n-1}})}=2a_{\Delta(C_{p^{n-1}})}$. Thus
\[pa_{\Delta(C_{p^{n-1}})}^2+p\sum_{Y}a_Y^2+\sum_{Z}a_Z^2=2a_{\Delta(C_{p^{n-1}})}\geq 0,\]
which implies that $a_{\Delta(C_{p^{n-1}})}\geq 0$. However, this further implies that $\frac{2}{p}\geq a_{\Delta(C_{p^{n-1}})}$. Since $p$ is an odd prime, this forces $a_{\Delta(C_{p^{n-1}})}=0$.
Which forces $a_X=0$ as $X$ runs over all elements of $\sS_{G\times G}$ that can be encoded as $(n-1,0;c,d)_\varphi$.

Dually, since $\alpha\alpha^\circ=\alpha^\circ \alpha$, we also get that $a_X=0$ as $X$ runs over elements $\sS_{G\times G}$ that can be encoded as $(c,d;n-1,0)_\varphi$.

Now we assume that for $X$ encoded as $(b,0;c,d)_\varphi$ or $X$ encoded as $(c,d;b,0)_\psi$, for $1<b\leq n-1$, we have $a_X=0$. Consider now the coefficient $c_{\Delta(C_{p^{b-1}})}$.
As before Proposition~\ref{explicitStarComputation} implies that if $X*Y=(b-1,0;b-1,0)=\Delta(C_{p^{b-1}})$ such that $X$ is encoded as $(b-1,0;c,d)_\varphi$ and $Y$ as $(c,d;,b-1,0)_{\psi}$, then
$X^\circ=Y$. Thus if we compute the coefficient $c_{\Delta(C_{p^{b-1}})}$ using the computation $\alpha\alpha^\circ$, we have
\[p^{n-b+1}a_{\Delta(C_{p^{b-1}})}^2+p^{n-b+1}\sum_{Z_{n-b+1}}a_{Z_{n-b+1}}^2+\cdots+p\sum_{Z_{1}}a_{Z_{1}}^2+\sum_{Z_{0}}a_{Z_{0}}^2=2a_{\Delta(C_{p^{b-1}})}\geq0,\]
where $Z_{n-b+1}$ runs over elements of $\sS_{G\times G}$ which can be encoded by $(b-1,0;b-1,0)_\varphi$ with $\varphi$ nontrivial, and $Z_{i}$ runs through all elements of $\sS_{G\times G}$
which can be encoded as $(b-1,0;c,d)_\psi$, with $c=i$, for $i=0,\dots, n-b$. As in the base case, we must have $\frac{2}{p^{n-b+1}}\geq a_{\Delta(C_{p^{b-1}})}\geq 0$.
This forces $a_{\Delta(C_{p^{b-1}})}=0$ and thus $a_{Z_i}=0$ as we run over all possible $Z_i$ and $i=0,\cdots,n-b+1$.

Considering $\alpha\alpha^\circ=\alpha^\circ \alpha$, inductively speaking we have that $a_X=0$ as $X$ runs through elements of $\sS_{G\times G}$ that can be encoded as $(i,0;c,d)_\varphi$
or $(c,d;i,0)_\psi$ for $i=1,\cdots, n$.

The final step is to consider the coefficients $a_X$ where $X$ is encoded as $(0,0;c,c)$ or $(d,d;0,0)$ for some $0\leq c,d\leq n$ (note we leave off the isomorphism, since it is trivial in this case). 
We show these coefficients are all $0$ as well. To do this, we compute the coefficient $a_{\Delta(\{1\})}$. However, there is a catch! We have proven so far that $a_X=0$ if $P_i(X)\neq K_i(X)=\{1\}$ for $i=1$ or $i=2$,
so if we consider elements $X,Y\in \sS_{G\times G}$ encoded as $(0,0;c,c)$ or $(d,d;,0,0)$, then $X*Y=(0,0;0,0)=\Delta(\{1\})$, as long as $X$ encoded as $(0,0;c,c)$ and $Y$ encoded
as $(d,d;0,0)$ for any $0\leq c,d\leq n$. Moreover, if $X$ is encoded as $(0,0;c,c)$ then $X^\circ$ is encoded as $(c,c;0,0)$. We abbreviate the coefficient
for $(0,0;c,c)$ in $\alpha$ by $a_c$, for all $0\leq c\leq n$. Hence, computing the coefficient $c_{\Delta(\{1\})}$
in $\alpha\alpha^\circ$ gives us
\[p^na_0^2+p^{n-1}a_1^2+\cdots+a_n^2+2\sum_{(i,j)}p^{n-i}a_ia_j\]
as $(i,j)$ runs over pairs $0\leq j<i\leq n$. Thus
\[p^na_0^2+p^{n-1}a_1^2+\cdots+a_n^2+2\sum_{(i,j)}p^{n-i}a_ia_j\]
\[=(a_0+\cdots+a_n)^2+(p^n-1)a_0^2+(p^{n-1}-1)a_1^2+\cdots+(p-1)a_{n-1}^2\]
\[+2\sum_{(l,k)}(p^{n-l}-1)a_la_k\]
\[=(a_0+\cdots+a_n)^2+(p-1)(a_0+\cdots+a_{n-1})^2\]
\[(p^n-p)a_0^2+(p^{n-1}-p)a_1^2+\cdots+(p-p)a_{n-2}^2\]
\[+2\sum_{(r,s)}(p^{n-r}-p)a_ra_s,\]
as $(l,k)$ runs over pairs $0\leq k<l<n$ and $(r,s)$ runs over pairs $0\leq s<r<n-1$

Continuing in this fashion, we get
\[c_{\Delta(\{1\})}=\left(\sum_{k=0}^na_k\right)^2+\sum_{i=0}^{n-1}(p^{n-i}-p^{n-1-i})\left(\sum_{j=0}^ia_j\right)^2.\]
However, we still have that $c_{\Delta(\{1\})}=2a_{\Delta(\{1\})}$. This implies that $a_{\Delta(\{1\})}\geq0$. Moreover, subtracting $(p^n-p^{n-1})a_{\Delta(\{1\})}^2$ from both sides of the equation
\[\left(\sum_{k=0}^na_k\right)^2+\sum_{i=0}^{n-1}(p^{n-i}-p^{n-1-i})\left(\sum_{j=0}^ia_j\right)^2=2a_{\Delta(\{1\})}\]
still leaves the left hand side nonnegative. Thus $2a_{\Delta(\{1\})}-(p^n-p^{n-1})a_{\Delta(\{1\})}^2\geq0\implies \frac{2}{p^{n}-p^{n-1}}\geq a_{\Delta(\{1\})}\geq0$. Since
$p$ is a prime number, we have $a_{\Delta(\{1\})}=0$. Thus the we have $\sum_{j=0}^ia_j=0$ for all $i=0,\dots n$, which implies $a_0=a_1=\cdots=a_n=0$.

Finally, if we repeat the symmetric argument for the product $\alpha^\circ \alpha$, we get that all the coefficients $a_X=0$ for $X\in\sS_{G\times G}$ encoded as $(0,0;c,c)$ or
$(d,d;0,0)$ for any $0\leq c,d\leq n$.  This proves that $\ker(\rho^\times)=\mathrm{im}(\dBInf^{C^{p^n}}_{C^{p^n}/C^{p}})$.

\end{proof}

\begin{proof}[Proof of Theorem~\ref{main2}]
Assume $G$ is cyclic of order $p^n$, where $p$ is a prime.

If $G$ is trivial, then $B(G,G)\cong B(G)\cong \bZ$, and $B(G,G)^\times=B_\circ(G,G)=\{\pm1\}$. If $n=1$, we are done by Proposition~\ref{baseCase}. 

Assume now that $p$ is an odd prime. For $n=k+1$, with $k\geq 1$, Proposition~\ref{inductiveCase} tells us that
\[B_\circ(G,G)=\mathrm{im}(\eta)\times\mathrm{im}(\dBInf_{C_{p^n}/C_p}^{C_{p^n}}).\]
By Lemma~\ref{workhorse} $(ii)$, $\mathrm{im}(\eta)\cong C_2\times\mathrm{Out}(G)$ and by induction we have,
\[\mathrm{im}(\dBInf_{C_{p^n}/C_p}^{C_{p^n}})\cong B_\circ(C_{p^{n-1}},C_{p^{n-1}})\cong \left\{\begin{array}{lll}
	C_2^{n+1}\times\prod_{i=1}^{n-1}\mathrm{Out}(C_{p^i}) &\mathrm{if}& p=3\\\\
	C_2^{n}\times\prod_{i=1}^{n-1}\mathrm{Out}(C_{p^i}) &\mathrm{if}& p>3
	\end{array}\right..\]
This finishes the proof.
\end{proof}

\end{document}